\documentclass[11pt,reqno]{amsart}
\usepackage{indentfirst, amssymb,amsmath,amsthm, mathrsfs,cite,graphicx,float,caption,subcaption}
\newtheorem*{theoA}{Theorem A}
\newtheorem*{theoB}{Theorem B}
\newtheorem*{theoC}{Theorem C}

\newtheorem{theo}{Theorem}[section]
\newtheorem{lem}{Lemma}[section]
\newtheorem{cor}{Corollary}[section]

\newtheorem{rem}{Remark}[section]

\newtheorem{defi}{Definition}

\newtheorem{open problem}{Open problem}[section]

\newcommand{\ol}{\overline}
\newcommand{\be}{\begin{equation}}
\newcommand{\ee}{\end{equation}}
\newcommand{\bs}{\begin{small}}
\newcommand{\es}{\end{small}}
\newcommand{\beas}{\begin{eqnarray*}}
\newcommand{\eeas}{\end{eqnarray*}}
\newcommand{\bea}{\begin{eqnarray}}
\newcommand{\eea}{\end{eqnarray}}
\renewcommand{\epsilon}{\varepsilon}
\numberwithin{equation}{section}
\begin{document}
\title[Geometric analysis ]{Geometric analysis of a class of harmonic mappings defined by a differential Inequality}
\author[V. Allu, R. Biswas and R. Mandal ]{Vasudevarao Allu, Raju Biswas and Rajib Mandal}
\date{}
\address{Vasudevarao Allu, Department of Mathematics, School of Basic Sciences, Indian Institute of Technology Bhubaneswar, Bhubaneswar-752050, Odisha, India.}
\email{avrao@iitbbs.ac.in}
\address{Raju Biswas, Department of Mathematics, Raiganj University, Raiganj, West Bengal-733134, India.}
\email{rajubiswasjanu02@gmail.com}
\address{Rajib Mandal, Department of Mathematics, Raiganj University, Raiganj, West Bengal-733134, India.}
\email{rajibmathresearch@gmail.com}
\let\thefootnote\relax
\footnotetext{2020 Mathematics Subject Classification: 30C45, 30C50, 30C80, 31A05.}
\footnotetext{Key words and phrases: Harmonic functions, starlike and convex functions, convolution, convex combination, coefficient estimate, growth theorem.}
\begin{abstract}
In this paper, we introduces and undertake as a systematical investigation of the class $\mathcal{P}_{\mathcal{H}}^{0}(\alpha,M)$ of normalized harmonic mappings $f = h + \overline{g}$ in the unit disk $\mathbb{D}$, defined by the differential inequality
\beas
\text{Re}\left((1-\alpha)h'(z) + \alpha z h''(z)\right) > -M + \left|(1-\alpha)g'(z) + \alpha z g''(z)\right|\quad\text{for}\quad z\in\Bbb{D},
\eeas
where $M > 0$, $\alpha \in (0,1]$, and $g'(0) = 0$. This class extends the harmonic analogue of functions with positive real part and offers a unified framework for analyzing their geometric characteristics. We obtain sharp coefficient bounds for both the analytic and co-analytic parts, establish sharp growth bounds, and determine the radii of univalency, starlikeness, and convexity. Furthermore, we show that $\mathcal{P}_{\mathcal{H}}^{0}(\alpha,M)$ is closed under convex combinations, and under suitable restrictions on the  parameters, it is also closed under convolution. Our findings generalize and extend several known
results in the theory of harmonic mappings.
\end{abstract}
\maketitle
\section{Introduction and Preliminary Results}
\noindent The application of harmonic mappings has become a valuable tool in the study of fluid flow problems (see \cite{AC2014}). Moreover, univalent harmonic functions that are characterized by geometric properties such as convexity, starlikeness, and close-to-convexity emerge naturally in the context of planar fluid dynamics problems.\\[2mm]
\indent Let $f=u+iv$ be a complex-valued function of $z=x+i y$ in a simply connected domain $\Omega$. If $f\in C^2(\Omega)$ (continuous first and second partial derivatives exist in $\Omega$) and  satisfies the Laplace equation $\Delta f =4f_{z\ol z} = 0$ in $\Omega$, 
then $f$ is said to be harmonic in $\Omega$. Note that every harmonic mapping $f$ has the canonical representation $f = h + \ol g$, where $h$ and 
$g$ are analytic in $\Omega$, known respectively as the analytic and co-analytic parts of $f$. This representation is 
unique up to an additive constant (see \cite{D2004}). The inverse function theorem and a result of Lewy \cite{L1936} shows that a harmonic function $f$ is locally univalent in 
$\Omega$ if, and only if, the Jacobian  of $f$, defined by $J_f(z):=|h'(z)|^2-|g'(z)|^2$ is non-zero in $\Omega$. 
A harmonic mapping $f$ is locally univalent and sense-preserving in $\Omega$ if, and only if, $J_f (z) > 0$ in $\Omega$. Equivalently, this condition holds if $h'\ne 0$ in $\Omega$ and the dilatation of $f$, defined as $\omega_f:= \omega=g'/h'$, satisfies $|\omega_f| < 1$ in $\Omega$ (see \cite{L1936}). 
Let $\mathcal{H}$ denote the class of all complex-valued harmonic functions $f=h+\ol g$ defined in $\mathbb{D}$, where $h$ and $g$ are analytic in $\mathbb{D}$ with the normalization
$h(0)=h'(0)-1=0$ and $g(0)=0$. If the co-analytic part $g(z)\equiv 0$ in $\mathbb{D}$, then the class $\mathcal{H}$ reduces to the class $\mathcal{A}$ of analytic functions in 
$\mathbb{D}$ with $f(0)=f'(0)-1=0$.
Let $\mathcal{S}_{\mathcal{H}}$ denote the class of functions $f\in\mathcal{H}$ that are both sense-preserving and univalent in $\mathbb{D}$ and let 
$\mathcal{S}_{\mathcal{H}}^0=\left\{f=h+\ol g\in\mathcal{S}_{\mathcal{H}}: g'(0)=0 \right\}$. Every $f=h+\ol g\in\mathcal{S}_{\mathcal{H}}^0$ has the following form:
\bea\label{e1}f(z)=h(z)+\ol {g(z)}=z+\sum_{n=2}^\infty a_nz^n+\ol{\sum_{n=2}^\infty b_nz^n}. \eea
\indent If the co-analytic part $g(z)\equiv 0$ in $\mathbb{D}$, then both the classes $\mathcal{S}_{\mathcal{H}}$ and $\mathcal{S}_{\mathcal{H}}^0$ reduces to the class $\mathcal{S}$ of 
univalent and analytic functions in $\mathbb{D}$ with $f(0)=f'(0)-1=0$. 
Both $\mathcal{S}_{\mathcal{H}}$ and $\mathcal{S}_{\mathcal{H}}^0$ are natural harmonic 
generalizations of $\mathcal{S}$, but only $\mathcal{S}_{\mathcal{H}}^0$ is compact although both the classes $\mathcal{S}_{\mathcal{H}}$ and $\mathcal{S}_{\mathcal{H}}^0$ are normal. In 1984, Clunie and Sheil-Small \cite{CS1984} undertook a comprehensive study of the class $\mathcal{S}_{\mathcal{H}}$ and its geometric subclasses. This study 
has subsequently garnered extensive attention from researchers (see \cite{BJJ2013, KPV2014,WL2001,MBG2024,BM2025, ABM2025,BM2026}). \\[2mm]
\indent A domain $\Omega$ is called starlike with respect to a point $z_0\in\Omega$ if the closed line segment joining $z_0$ to any point in $\Omega$ lies in $\Omega$, {\it i.e.}, 
$(1-t)z_0+t z\in\Omega$ for $t\in[0,1]$ and $z\in\Omega$. In particular, if $z_0=0$, then $\Omega$ is simply called starlike. A complex-valued harmonic mapping 
$f\in\mathcal{H}$ is said to be starlike if $f(\mathbb{D})$ is starlike. 
Let $\mathcal{S}_{\mathcal{H}}^*$ denote the class of harmonic starlike functions in $\mathbb{D}$. A domain $\Omega$ is called convex if it is starlike with respect to any point 
in $\Omega$. In other words, convexity implies starlikeness, but the converse is not necessarily true. A domain can be starlike 
without being convex. A function $f\in\mathcal{H}$ is said to be convex if $f(\mathbb{D})$ is convex. Let  $\mathcal{K}_{\mathcal{H}}$ denote the class of harmonic convex 
mappings in $\mathbb{D}$.  A domain $\Omega$ is said to be close-to-convex if the complement of $\Omega$ can be expressed as a union of closed half-lines with the condition that the corresponding open half-lines are disjoint. A function $f\in\mathcal{S}_{\mathcal{H}}$ which maps the unit disk $\mathbb{D}$ onto a close-to-convex domain is called a close-to-convex function.  The following definitions are essential for a comprehensive understanding of this paper.
\begin{defi}\cite{D1983,SS1989} Let $\psi_1$ and $\psi_2$ be two analytic functions in $\mathbb{D}$ given by $\psi_1(z)=\sum_{n=0}^{\infty}a_nz^n$ and $\psi_2(z)=\sum_{n=0}^{\infty}b_nz^n$. The convolution (or, Hadamard product) is defined by $\left(\psi_1*\psi_2\right)(z)=\sum_{n=0}^{\infty}a_nb_nz^n~\text{for}~ z\in\mathbb{D}$. \end{defi}
\begin{defi}\cite{G2002} For harmonic functions $f_1=h_1+\ol{g_1}$ and $f_2=h_2+\ol{g_2}$ in $\mathcal{H}$, the convolution is defined as $f_1*f_2=h_1*h_2+\ol{g_1*g_2}$.\end{defi}
\begin{defi}\cite{SS1989} A sequence $\{a_n\}$ of non-negative numbers is said to be a convex null sequence if $a_n\to 0$ as $n\to\infty$ and $a_0-a_1\geq a_1-a_2\geq \cdots\geq a_{n-1}-a_n\geq\cdots\geq 0$.
\end{defi}
\noindent Let $\mathcal{P}$ denote the class of analytic functions $h$ in $\mathbb{D}$ that satisfy the condition $\textrm{Re}\left(h'(z)\right)>0$ in $\mathbb{D}$ with $h(0)=h'(0)-1=0$.
It is well-known that $\mathcal{P}\subsetneq\mathcal{S}$. 
MacGregor \cite{M1962} proved that each partial sum $s_n(h)=\sum_{k=0}^n a_k z^k$ ($n\geq 2$) of the function $h(z)=\sum_{k=0}^\infty a_k z^k\in\mathcal{P}$ maintained its univalence within the disk $|z|<1/2$. Furthermore, $h(z)$ maps the disk $|z|<\sqrt{2}-1$ onto a convex domain. The numbers $1/2$ and $\sqrt{2}-1$ are best possible constants. Afterwards, Singh \cite{S1970} proved that each partial sum $s_n(h)$ of the function $h(z)=\sum_{k=0}^\infty a_k z^k\in\mathcal{P}$ is convex in the disk $|z| < 1/4$. The number $1/4$ is the best possible constant. \\[2mm]
\indent Ponnusamy {\it et al.} \cite{PYY2013} have studied the following class as a harmonic analogue of the class $\mathcal{P}$: 
\beas\mathcal{P}_{\mathcal{H}}:=\left\{f=h+\ol{g}\in\mathcal{H}: \textrm{Re}\left(h'(z)\right)>|g'(z)|\quad\text{in}\quad\mathbb{D}\right\} \eeas
and $\mathcal{P}_{\mathcal{H}}^0:=\left\{f=h+\ol{g}\in\mathcal{P}_{\mathcal{H}}: g'(0)=0\right\}$. The authors \cite{PYY2013} proved that functions in $\mathcal{P}_{\mathcal{H}}$ are close-to-convex in $\mathbb{D}$. Li and Ponnusamy \cite{1LP2013, 2LP2013} have investigated the radius of univalency and convexity of sections of functions $f\in \mathcal{P}_{\mathcal{H}}^0$.\\[2mm]
\indent In 2020, Ghosh and Allu \cite{GA2020} introduced the class $\mathcal{P}_{\mathcal{H}}^0(M)$ for $M>0$, of all functions $f=h+\ol{g}\in\mathcal{H}$ satisfying the following conditions:
\beas\text{Re}\left(zh''(z)\right)>-M+|zg''(z)|~\text{with}~ g'(0)=0~\text{for}~ z\in\mathbb{D}\eeas
and established the following results regarding the sharp coefficient bounds and growth results for functions in $\mathcal{P}_{\mathcal{H}}^0(M)$.
\begin{theoA}\cite{GA2020} Let $f=h+\ol{g}\in \mathcal{P}_{\mathcal{H}}^0(M)$ for $M >0$ be of the form (\ref{e1}). Then, for $n\geq 2$, $|b_n|\leq 2M/(n(n-1))$. The result is sharp for the function $f(z)= z-M\ol{z}^n/(n(n-1))$.\end{theoA} 
\begin{theoB}\cite{GA2020}
Let $f=h+\ol{g}\in \mathcal{P}_{\mathcal{H}}^0(M)$ for $M >0$ be of the form (\ref{e1}). Then, for $n\geq 2$, 
(i)  $|a_n|+|b_n|\leq \frac{2M}{n(n-1)}$; (ii) $\left||a_n|-|b_n|\right|\leq \frac{2M}{n(n-1)}$;
(iii)$|a_n|\leq\frac{2M}{n(n-1)}$.
The results are sharp for the function $f$ given by $f'(z)=1-2M\ln(1-z)$.
\end{theoB}
\begin{theoC} \cite{GA2020} Let $f=h+\ol{g}\in \mathcal{P}_{\mathcal{H}}^0(M)$ for $M >0$ be of the form (\ref{e1}). Then
\beas |z|-2M\sum_{n=2}^{\infty}\frac{|z|^{n}}{n(n-1)}\leq |f(z)|\leq |z|+2M\sum_{n=2}^{\infty}\frac{|z|^n}{n(n-1)}.\eeas
The right-hand inequality is sharp for the function $f$ given by $f'(z)=1-2M\ln(1-z)$.
\end{theoC}
\noindent Motivated by the results of Ghosh and Allu \cite{GA2020}, we consider the class $\mathcal{P}_{\mathcal{H}}^0(\alpha,M)$ of all functions $f=h+\ol{g}\in\mathcal{H}$ satisfying the following conditions:
\beas\text{Re}\left((1-\alpha)h'(z)+ \alpha zh''(z)\right)>-M+\left|(1-\alpha)g'(z)+\alpha zg''(z)\right|~\text{and}~ g'(0)=0\eeas
for $M>0$, $\alpha\in(0,1]$ and $z\in\mathbb{D}$. \\[2mm]
\indent The parameter $\alpha \in (0,1]$ in the definition of $\mathcal{P}_{\mathcal{H}}^0(\alpha,M)$ provides a natural connection between different geometric behaviours. The operator $(1-\alpha)h'(z) + \alpha zh''(z)$ can be interpreted geometrically as a linear combination of the derivative (related to starlikeness) and the second-order operator $zh''(z)$ (related to convexity). When $\alpha = 1$, we recover the class $\mathcal{P}_{\mathcal{H}}^0(M)$ studied by Ghosh and Allu \cite{GA2020},  whereas varying $\alpha$ allows us to investigate a continuous family of classes that connect different geometric properties. This generalization provides a more in-depth understanding of how the local behaviour of harmonic mappings, determined by their derivatives, affects their global geometric properties, such as univalency, starlikeness, and convexity.\\ [2mm]
The organization of this paper is:
In section $2$, we establish the sharp coefficients bounds and growth estimate for functions in $\mathcal{P}_{\mathcal{H}}^0(\alpha,M)$. Furthermore, 
 we conclude the section by determining the radii of univalency, starlikeness, and convexity for functions in the class $\mathcal{P}_{\mathcal{H}}^0(\alpha,M)$. 
 Section $3$ is devoted to the structural properties of the class $\mathcal{P}_{\mathcal{H}}^{0}(\alpha, M)$, specifically under convex combinations and various convolution operations.
\section{The coefficient bounds, growth estimate and Geometric Radii}
\noindent In the following result, we establish the sharp coefficient bounds for the co-analytic part of functions in the class $\mathcal{P}_{\mathcal{H}}^0(\alpha,M)$.
\begin{theo}\label{T2} Let $f=h+\ol{g}\in\mathcal{P}_{\mathcal{H}}^0(\alpha,M)$ be of the form (\ref{e1}) for $M>0$ and $\alpha\in(0,1]$. For $n\geq 2$, we have
\beas |b_n|\leq \frac{M-\alpha+1}{n+\alpha n(n-2)}.\eeas
The result is sharp for the function $f$ given by $f(z)=z+\frac{M-\alpha+1}{n+\alpha n(n-2)}\ol{z^n}$. \end{theo}
\begin{proof} Let $f=h+\ol{g}\in \mathcal{P}_{\mathcal{H}}^0(\alpha,M)$. Therefore, we have
\bea\label{e2} \left|(1-\alpha)g'(z)+\alpha zg''(z)\right|<M+\text{Re}\left((1-\alpha)h'(z)+\alpha zh''(z)\right)\;\;\text{for} \;z\in\mathbb{D}.\eea
Since $(1-\alpha)g'(z)+\alpha zg''(z)=\sum_{n=2}^\infty\left(n+\alpha n(n-2)\right)b_nz^{n-1}$ is analytic in $\mathbb{D}$, in view of the Cauchy's integral formula for derivatives, we have 
\beas \left(n+\alpha n(n-2)\right)b_n=\frac{1}{2\pi i}\int_{|z|=r}\frac{(1-\alpha)g'(z)+\alpha zg''(z)}{z^{n}}\;dz \quad \text{for any}\quad r<1.\eeas
Thus, we have
\bea\left(n+\alpha n(n-2)\right)|b_n|&=&\left|\frac{1}{2\pi i}\int_0^{2\pi}\frac{(1-\alpha)g'(re^{i\theta})+\alpha re^{i\theta}g''(re^{i\theta})}{r^{n}e^{in\theta}}ir e^{i\theta}d\theta\right|\nonumber\\[2mm]
\label{eee2}&\leq&\frac{1}{2\pi}\int_0^{2\pi}\frac{\left|(1-\alpha)g'(re^{i\theta})+\alpha re^{i\theta}g''(re^{i\theta})\right|}{r^{n-1}}d\theta.\eea
By the Mean value property (see \cite{P2005, D2004}), if $\psi$ is analytic in a domain $\Omega\subseteq \Bbb{C}$ containing the closed disk $|z-z_0|\leq \rho$, then $\text{Re}(\psi)$ satisfies:
\bea\label{ee2} \text{Re}(\psi(z_0))=\frac{1}{2\pi}\int_0^{2\pi} \text{Re}(\psi(z_0+\rho e^{i\theta}))\;d\theta.\eea
Using (\ref{e2}) to (\ref{eee2}), we obtain
\beas\left(n+\alpha n(n-2)\right) r^{n-1}|b_n|&\leq&\frac{1}{2\pi}\int_0^{2\pi}\left(M+\text{Re}\left((1-\alpha)h'(re^{i\theta})+\alpha re^{i\theta}h''(re^{i\theta})\right)\right)\;d\theta\\[2mm]
&=&M+\frac{1}{2\pi}\int_0^{2\pi}\text{Re}\left((1-\alpha)h'(re^{i\theta})+\alpha re^{i\theta}h''(re^{i\theta})\right)\;d\theta.\eeas
Applying the mean value property (\ref{ee2}) to the analytic function $(1-\alpha)h'(z)+\alpha zh''(z)=(1-\alpha)+\sum_{n=2}^\infty (n+\alpha n(n-2))a_n z^{n-1}$, we obtain
\beas \frac{1}{2\pi}\int_0^{2\pi}\text{Re}\left((1-\alpha)h'(re^{i\theta})+\alpha re^{i\theta}h''(re^{i\theta})\right)d\theta = 1-\alpha.\eeas
Therefore,
\beas \left(n+\alpha n(n-2)\right) r^{n-1}|b_n| \leq M + (1-\alpha) = M - \alpha + 1.\eeas
Letting $r\to 1^-$ gives the desired bound.\\[2mm]
\indent To prove the sharpness of the result, we consider the function
 $$f(z)=h(z)+\ol{g(z)} = z + \frac{M-\alpha+1}{n+\alpha n(n-2)}\overline{z^n}\quad\text{for}\quad z\in\Bbb{D}\quad\text{and}\quad n\geq2.$$
Therefore, $h(z) = z$ and $g(z) =(M-\alpha+1)/(n+\alpha n(n-2))z^n$. Then, we obtain
\beas (1-\alpha)h'(z) + \alpha zh''(z)=1-\alpha\quad\text{and}\quad(1-\alpha)g'(z) + \alpha zg''(z) = (M-\alpha+1)z^{n-1}.\eeas
For $n\geq 2$ and $z\in\Bbb{D}$, it is evident that 
\beas -M + \left|(1-\alpha)g'(z) + \alpha zg''(z)\right| &=& -M + (M-\alpha+1)|z|^{n-1}\\
&=&(1-\alpha)|z|^{n-1}+M\left(|z|^{n-1}-1\right)\\
&<&1-\alpha=\text{Re}\left((1-\alpha)h'(z) + \alpha zh''(z)\right),\eeas
which shows that $f \in \mathcal{P}_{\mathcal{H}}^{0}(\alpha,M)$ and $|b_n(f)|=(M-\alpha+1)/(n+\alpha n(n-2))$.
This completes the proof.\end{proof}
\noindent Let $\mathcal{P}(\alpha,M)$ denote the class of functions $\phi\in \mathcal{A}$ that satisfy the following condition:
\beas \text{Re}\left((1-\alpha)\phi'(z)+ \alpha z\phi''(z)\right)>-M,~\text{where}~M >0, \alpha\in(0,1]~\text{and}~ z\in\mathbb{D}.\eeas
In the following result, we establish a correlation between the functions in $\mathcal{P}(\alpha,M)$ and $\mathcal{P}_{\mathcal{H}}^0(\alpha, M)$.
\begin{theo}\label{T1}
A harmonic mapping $f=h+\ol{g}$ belongs to $\mathcal{P}_{\mathcal{H}}^0(\alpha,M)$ if, and only if, the
function $F_\epsilon=h +\epsilon g\in\mathcal{P}(\alpha,M$) for each $\epsilon$ $(|\epsilon|=1)$.
\end{theo}
\begin{proof} Suppose that $f=h+\ol{g}\in \mathcal{P}_{\mathcal{H}}^0(\alpha,M)$. Thus, we have
\bea\label{es1}\text{Re}\left((1-\alpha)h'(z)+ \alpha zh''(z)\right)>-M+\left|(1-\alpha)g'(z)+\alpha zg''(z)\right|\;\text{for}\;z\in\mathbb{D}.\eea
Fix $|\epsilon|=1$. Since $F_\epsilon=h +\epsilon g$, we deduce by using (\ref{es1}) that
\beas\text{Re}\left((1-\alpha)F_{\epsilon}'(z)+ \alpha zF_{\epsilon}''(z)\right)
&=&\text{Re}\left(\left((1-\alpha)h'(z)+ \alpha zh''(z)\right)\right.\\[2mm]&&\left.+\epsilon\left((1-\alpha)g'(z)+\alpha zg''(z)\right)\right)\\[2mm]
&\geq& \text{Re}\left((1-\alpha)h'(z)+ \alpha zh''(z)\right)\\[2mm]&&-\left|(1-\alpha)g'(z)+\alpha zg''(z)\right|\\[2mm]
&>&-M \quad\text{for} \quad z\in\mathbb{D}. \eeas
Therefore, $F_\epsilon=h +\epsilon g\in\mathcal{P}(\alpha,M)$ for each $\epsilon$ with $|\epsilon|=1$.
Conversely, if $F_\epsilon\in\mathcal{P}(\alpha,M)$, then for $z\in\mathbb{D}$, we have
\beas&& \qquad\text{Re}\left((1-\alpha)F_{\epsilon}'(z)+ \alpha zF_{\epsilon}''(z)\right)>-M,\\[2mm]
&&\text{\it i.e.,}\quad\text{Re}\left(\left((1-\alpha)h'(z)+ \alpha zh''(z)\right)+\epsilon\left((1-\alpha)g'(z)+\alpha zg''(z)\right)\right)>-M,\\[2mm]
&&\text{\it i.e.,}\quad\text{Re}\left((1-\alpha)h'(z)+ \alpha zh''(z)\right)>-M+\text{Re}\left(-\epsilon\left((1-\alpha)g'(z)+\alpha zg''(z)\right)\right). \eeas
Let $p(z) = (1-\alpha)g^{\prime}(z)+\alpha zg^{\prime\prime}(z)$ and write $p(z) = |p(z)| e^{i\theta(z)}$. Since $|\epsilon|=1$, we can choose $\epsilon = -e^{-i\theta(z)}$. Then, we have
\beas -\epsilon p(z) = e^{-i\theta(z)} \cdot |p(z)| e^{i\theta(z)} = |p(z)|,\eeas
which shows that $\text{Re}(-\epsilon p(z)) = |p(z)|$. Substituting this specific $\epsilon$ into the inequality gives
\beas \text{Re}\left((1-\alpha)h'(z)+ \alpha zh''(z)\right)>-M+\left|(1-\alpha)g'(z)+\alpha zg''(z)\right|\quad\text{for} \quad z\in\mathbb{D}, \eeas
which shows that $f\in\mathcal{P}_{\mathcal{H}}^0(\alpha,M)$. This completes the proof.\end{proof}
\noindent In the following result, we establish the sharp coefficient bounds for functions in the class $\mathcal{P}_{\mathcal{H}}^0(\alpha,M)$.
\begin{theo}\label{T3} Let $f=h+\ol{g}\in\mathcal{P}_{\mathcal{H}}^0(\alpha,M)$ be of the form (\ref{e1}) for $M>0$ and $\alpha\in(0,1]$. For $n\geq 2$, we have 
\begin{enumerate}
\item[(i)]  $|a_n|+|b_n|\leq \dfrac{2(M-\alpha+1)}{n+\alpha n(n-2)}$;
\item[(ii)] $\left||a_n|-|b_n|\right|\leq \dfrac{2(M-\alpha+1)}{n+\alpha n(n-2)}$;
\item[(iii)] $|a_n|\leq \dfrac{2(M-\alpha+1)}{n+\alpha n(n-2)}$.\end{enumerate}
The results are sharp for the function $f$ given by $f(z)=z+\sum_{n=2}^\infty \frac{2(M-\alpha+1)}{n+\alpha n(n-2)}z^n$.\end{theo}
\begin{proof} Let $f=h+\ol{g}\in \mathcal{P}_{\mathcal{H}}^0(\alpha,M)$. In view of \textrm{Theorem \ref{T1}}, we have $F_{\epsilon}=h+\epsilon g\in  \mathcal{P}(\alpha,M)$ for each $\epsilon$ $(|\epsilon|=1)$. Therefore, we have
\beas \text{Re}\left((1-\alpha)F_{\epsilon}'(z)+\alpha zF_{\epsilon}''(z)+M\right)>0 \quad \text{for}\quad z\in\mathbb{D}.\eeas
Thus there exists an analytic function $p(z)=1+\sum_{n=1}^\infty p_nz^n$ in $\mathbb{D}$ with $\text{Re}\;p(z)>0$ in $\mathbb{D}$ and $p(0)=1$ such that
\bea&&\frac{(1-\alpha)F_{\epsilon}'(z)+\alpha zF_{\epsilon}''(z)+M}{M-\alpha+1}=p(z)\nonumber\\[1mm]\text{\it i.e.,}
&&\frac{(M-\alpha+1)+\sum_{n=2}^{\infty}\left(n+\alpha (n^2-2n)\right)\left(a_n+\epsilon b_n\right)z^{n-1}}{M-\alpha+1}=1+\sum_{n=1}^\infty p_nz^n\nonumber\\[1mm]\text{\it i.e.,}
&&\label{e3}\sum_{n=1}^{\infty}\left(n+1+\alpha (n^2-1)\right)\left(a_{n+1}+\epsilon b_{n+1}\right)z^n=(M-\alpha+1)\sum_{n=1}^\infty p_nz^n.\eea
Comparing coefficients on both sides of (\ref{e3}), we have 
\bea\label{e4}\left(n+1+\alpha (n^2-1)\right)\left(a_{n+1}+\epsilon b_{n+1}\right)=(M-\alpha+1)p_n\;\;\text{for}\;n\geq 1.\eea
By the well-known coefficient estimates for functions with positive real parts (see \cite{C1907,TTA2018}), we have  $|p_n|\leq 2$ for $n\geq 1$. Given that $\epsilon$ $(|\epsilon|=1)$ is arbitrary, it follows from (\ref{e4}) that
\beas &&\left(n+1+\alpha (n^2-1)\right)(|a_{n+1}|+|b_{n+1}|)\leq 2(M-\alpha+1)\\[2mm]\text{and}
&&\left(n+1+\alpha (n^2-1)\right)||a_{n+1}|-|b_{n+1}||\leq 2(M-\alpha+1).\eeas
Therefore, for $n\geq 2$, we have
\beas |a_{n}|+|b_{n}|\leq \frac{2(M-\alpha+1)}{n+\alpha (n^2-2n)}, ||a_{n}|-|b_{n}||\leq \frac{2(M-\alpha+1)}{n+\alpha (n^2-2n)}\;\text{and}\;|a_n|\leq \frac{2(M-\alpha+1)}{n+\alpha (n^2-2n)}.\eeas
To prove the sharpness of these bounds, we consider the function $f$ defined by 
\bea\label{extremal}f(z)=h(z)+\ol{g(z)}=z+\sum_{n=2}^\infty \frac{2(M-\alpha+1)}{n+\alpha (n^2-2n)}z^n,\quad z\in\Bbb{D}.\eea
It is clear that $g(z)\equiv 0$ for $z\in\Bbb{D}$ and for $n\geq 2$, we have 
\beas a_n = \frac{2(M-\alpha+1)}{n+\alpha n(n-2)} \quad \text{and}\quad b_n = 0.\eeas
Evidently,
\beas
(1-\alpha)h'(z) + \alpha zh''(z)&=&
(1-\alpha) + \sum_{n=2}^{\infty} 2(M-\alpha+1) z^{n-1} \\
&=& (1-\alpha) + 2(M-\alpha+1) \frac{z}{1-z}.\eeas
Since $\text{Re}\left(\dfrac{z}{1-z}\right) > -\dfrac{1}{2}$ for $z \in \mathbb{D}$, we have
\beas \text{Re}\left((1-\alpha)h'(z)+\alpha zh''(z)\right) > (1-\alpha) - (M-\alpha+1) = -M,\eeas
which shows that $f \in \mathcal{P}^{0}_{\mathcal{H}}(\alpha,M)$. 
For the function $f$ given in (\ref{extremal}), we observe that
\begin{align*}
|a_n| + |b_n| &= \frac{2(M-\alpha+1)}{n+\alpha n(n-2)} + 0 = \frac{2(M-\alpha+1)}{n+\alpha n(n-2)}, \\
\left| |a_n| - |b_n| \right| &= \left| \frac{2(M-\alpha+1)}{n+\alpha n(n-2)} - 0 \right| = \frac{2(M-\alpha+1)}{n+\alpha n(n-2)}, \\
|a_n| &= \frac{2(M-\alpha+1)}{n+\alpha n(n-2)}.
\end{align*}
All bounds are attained with equality, thereby confirming their sharpness. This completes the proof.\end{proof}
\noindent These sharp coefficient bounds naturally lead to the question of the growth of functions within the class $\mathcal{P}_{\mathcal{H}}^0(\alpha,M)$. In the following result, we address this inquiry.
\begin{theo}\label{Theo1}  Let $f=h+\ol{g}\in\mathcal{P}_{\mathcal{H}}^0(\alpha,M)$ be of the form (\ref{e1}) for $M>0$ and $\alpha\in(0,1]$. Then, 
\be\label{eq1}|z|+2(M-\alpha+1)\sum_{n=2}^{\infty}\frac{(-1)^{n-1}|z|^{n}}{n(1+\alpha (n-2))}\leq |f(z)|\leq |z|+2(M-\alpha+1)\sum_{n=2}^{\infty}\frac{|z|^n}{n(1+\alpha (n-2))}.\ee
For each $z\in\mathbb{D}$, $z\not=0$, equality occurs for the function $f$ given by $f(z)=z+2(M-\alpha+1)\sum_{n=2}^{\infty}\frac{z^n}{n(1+\alpha (n-2))}$ or its suitable rotations.
\end{theo}
\begin{proof} Let $f=h+\ol{g}\in \mathcal{P}_{\mathcal{H}}^0(\alpha,M)$ for $M>0$ and $\alpha\in(0,1]$. Then, $F_{\epsilon}=h+\epsilon g\in  \mathcal{P}(\alpha,M)$ for each $\epsilon$ with $|\epsilon|=1$ and it follows that
\beas \text{Re}\left((1-\alpha)F_{\epsilon}'(z)+\alpha zF_{\epsilon}''(z)+M\right)>0 \;\;\text{for}\;\; z\in\mathbb{D}.\eeas
Thus, there exists an analytic function $\omega : \mathbb{D}\to \mathbb{D}$ with $\omega(0)=0$ such that
\bea\label{e5}&& \frac{(1-\alpha)F_{\epsilon}'(z)+\alpha zF_{\epsilon}''(z)+M}{M-\alpha+1}=\frac{1+\omega(z)}{1-\omega(z)},\nonumber\\[1mm]\text{\it i.e.,}
&& (1-\alpha)F_{\epsilon}'(z)+\alpha zF_{\epsilon}''(z)=(M-\alpha+1)\frac{1+\omega(z)}{1-\omega(z)}-M,\nonumber\\[1mm]\text{\it i.e.,}
&& z^{\frac{1}{\alpha}-2}\left((1-\alpha)F_{\epsilon}'(z)+\alpha zF_{\epsilon}''(z)\right)=(1-\alpha)z^{\frac{1}{\alpha}-2}\frac{1+\omega(z)}{1-\omega(z)}+2Mz^{\frac{1}{\alpha}-2}\frac{\omega(z)}{1-\omega(z)},\nonumber\\[1mm]\text{\it i.e.,}
&&\frac{d}{dz}\left(\alpha z^{\frac{1}{\alpha}-1}F_{\epsilon}'(z)\right)=(1-\alpha)z^{\frac{1}{\alpha}-2}\frac{1+\omega(z)}{1-\omega(z)}+2Mz^{\frac{1}{\alpha}-2}\frac{\omega(z)}{1-\omega(z)}.\eea
Note that $1/\alpha-1\geq 0$ and $z^{\frac{1}{\alpha}-1}=\mathrm{exp}((1/\alpha-1)\log(z))$, where the branch of the logarithm is determined by $\log(1)=0$. This ensure that the function is single-valued and analytic. Hence, the function $z^{\frac{1}{\alpha}-1}$ is differentiable (see \cite{P2005}).
Using the Schwarz Lemma, we have $|\omega(z)|\leq|z|$ for $z\in\mathbb{D}$. We consider the following cases.\\[1mm]
{\bf Case 1.} Let $\alpha\in(0,1)$. Using $F_{\epsilon}'(0)=1$, from (\ref{e5}), we obtain 
\bea\label{ew1}\alpha z^{\frac{1}{\alpha}-1}F_{\epsilon}'(z)=(1-\alpha)\int_{0}^z \xi^{\frac{1}{\alpha}-2}\frac{1+\omega(\xi)}{1-\omega(\xi)}\; d\xi+2M\int_{0}^z \xi^{\frac{1}{\alpha}-2}\frac{\omega(\xi)}{1-\omega(\xi)}\; d\xi.\eea
From (\ref{ew1}), we have 
\beas&&\left|\alpha z^{\frac{1}{\alpha}-1}F_{\epsilon}'(z)\right|\nonumber\\[2mm]
&=&\left|(1-\alpha)\int_{0}^{|z|} \left(te^{i\theta}\right)^{\frac{1}{\alpha}-2}\frac{1+\omega(te^{i\theta})}{1-\omega(te^{i\theta})} e^{i\theta}dt+2M\int_{0}^{|z|} \left(te^{i\theta}\right)^{\frac{1}{\alpha}-2}\frac{\omega(te^{i\theta})}{1-\omega(te^{i\theta})} e^{i\theta}dt\right|\nonumber\\[1mm]
&\leq &(1-\alpha)\int_{0}^{|z|}t^{\frac{1}{\alpha}-2}\frac{1+t}{1-t} dt+2M\int_{0}^{|z|}t^{\frac{1}{\alpha}-2}\frac{t}{1-t} dt\nonumber\\[1mm]
&\leq &(1-\alpha)\int_{0}^{|z|}\left(t^{\frac{1}{\alpha}-2}+2t^{\frac{1}{\alpha}-1}\sum_{n=0}^{\infty}t^n dt\right)+2M\int_{0}^{|z|}t^{\frac{1}{\alpha}-1}\sum_{n=0}^{\infty}t^n dt\nonumber\\[1mm]
&=&(1-\alpha)\frac{|z|^{\frac{1}{\alpha}-1}}{\frac{1}{\alpha}-1}+2(M-\alpha+1)\sum_{n=0}^{\infty}\frac{|z|^{\frac{1}{\alpha}+n}}{\frac{1}{\alpha}+n}.\nonumber\eeas
Therefore, we have
\bea\label{e6} |F_{\epsilon}'(z)|=|h'(z)+\epsilon g'(z)|&\leq& 1+2(M-\alpha+1)\sum_{n=0}^{\infty}\frac{|z|^{n+1}}{1+\alpha n}\nonumber\\
&=& 1+2(M-\alpha+1)\sum_{n=1}^{\infty}\frac{|z|^{n}}{1+\alpha (n-1)}.\eea
Since (\ref{e6}) holds for all $\epsilon$ with $|\epsilon|=1$, thus, we have
\beas |h'(z)|+|g'(z)|\leq 1+2(M-\alpha+1)\sum_{n=1}^{\infty}\frac{|z|^{n}}{1+\alpha (n-1)}. \eeas
Let $z\in\Bbb{D}$ and $\Gamma$ be the radial segment from $0$ to $z$. Therefore, we have 
\beas|f(z)|&=&\left|\int_{\Gamma}\left(\frac{\partial f}{\partial \xi}d\xi+\frac{\partial f}{\partial \ol{\xi}}d\ol{\xi}\right)\right|\\[1mm]
&=& \left|\int_\Gamma\left(h'(\xi)d\xi + \ol{g'(\xi)}d\ol{\xi}\right)\right|\\[1mm]
&\leq& \int_{\Gamma}\left(|h'(\xi)|+|g'(\xi)|\right)|d\xi|\\[1mm]
&=& \int_{0}^{|z|}\left(|h^{\prime}(te^{i\theta})| + |g^{\prime}(te^{i\theta})|\right)dt \\[1mm]
&\leq& \int_0^{|z|}\left(1+2(M-\alpha+1)\sum_{n=1}^{\infty}\frac{t^{n}}{1+\alpha (n-1)}\right)dt\\[1mm]
&=&|z|+2(M-\alpha+1)\sum_{n=2}^{\infty}\frac{|z|^n}{n(1+\alpha (n-2))}.\eeas
Applying (\ref{ew1}), we have 
\bea\left|\alpha z^{\frac{1}{\alpha}-1}F_{\epsilon}'(z)\right|&\geq & \text{Re}\left(\alpha z^{\frac{1}{\alpha}-1}F_{\varepsilon}^{\prime}(z)\right) \nonumber\\
&=& (1-\alpha)\int_{0}^{|z|}t^{\frac{1}{\alpha}-2}\text{Re}\left(\frac{1+\omega(te^{i\theta})}{1-\omega(te^{i\theta})}\right) \;dt\nonumber\\
\label{ew2}&&+2M\int_{0}^{|z|} t^{\frac{1}{\alpha}-2}\text{Re}\left(\frac{\omega(te^{i\theta})}{1-\omega(te^{i\theta})}\right)\; dt.\eea
Evidently,
\bea\label{e7}\text{Re}\left(\frac{1+\omega(z)}{1-\omega(z)}\right)&=&\frac{1}{2}\left(\frac{1+\omega(z)}{1-\omega(z)}+\frac{1+\ol{\omega(z)}}{1-\ol{\omega(z)}}\right)\nonumber\\[2mm]
&=&\frac{1-|\omega(z)|^2}{|1-\omega(z)|^2}\geq \frac{(1+|\omega(z)|)(1-|\omega(z)|)}{(1+|\omega(z)|)^2}\geq\frac{1-|z|}{1+|z|}.\eea
Furthermore, note that
\beas&&\text{Re}(\omega(z))\geq -|\omega(z)|,\quad\text{\it i.e.,}\quad  \text{Re}(\omega(z))-|\omega(z)|^2\geq -|\omega(z)|(1+|\omega(z)|)\\[2mm]\text{and}
&&\frac{\text{Re}(\omega(z))-|\omega(z)|^2}{|1-\omega(z)|^2}\geq \frac{-|\omega(z)|(1+|\omega(z)|)}{\left(1+|\omega(z)|\right)^2}=\frac{-|\omega(z)|}{1+|\omega(z)|}.\eeas
Hence, we have
\bea\label{r1} \text{Re}\left(\frac{\omega(z)}{1-\omega(z)}\right)&=&\frac{\text{Re}(\omega(z))-|\omega(z)|^2}{|1-\omega(z)|^2}\geq\frac{-|z|}{1+|z|}.\eea
From (\ref{ew2}), (\ref{e7}) and (\ref{r1}), we have
\beas&&|\alpha z^{\frac{1}{\alpha}-1}F_{\epsilon}'(z)|\\
&\geq& (1-\alpha)\int_{0}^{|z|}t^{\frac{1}{\alpha}-2}\frac{1-t}{1+t} dt-2M\int_{0}^{|z|}t^{\frac{1}{\alpha}-2}\frac{t}{1+t} dt\nonumber\\[1mm]
&=&(1-\alpha)\int_{0}^{|z|}\left(t^{\frac{1}{\alpha}-2}-2t^{\frac{1}{\alpha}-1}\sum_{n=0}^{\infty}(-1)^nt^n dt\right)-2M\int_{0}^{|z|}t^{\frac{1}{\alpha}-1}\left(\sum_{n = 0}^\infty (-1)^{n} t^n\right)\\[1mm]
&=&(1-\alpha)\frac{|z|^{\frac{1}{\alpha}-1}}{\frac{1}{\alpha}-1}-2(M-\alpha+1)\sum_{n=0}^{\infty}(-1)^n\frac{|z|^{\frac{1}{\alpha}+n}}{\frac{1}{\alpha}+n}.\eeas
Therefore, we have 
\bea\label{e81}|F_{\epsilon}'(z)|=|h'(z)+\epsilon g'(z)|\geq 1-2(M-\alpha+1)\sum_{n=1}^{\infty}(-1)^{n-1}\frac{|z|^{n}}{1+\alpha (n-1)}.\eea 
Since (\ref{e81}) holds for all $\epsilon$ with $|\epsilon|=1$, thus, we have
\bea\label{e9} |h'(z)|-|g'(z)|\geq 1-2(M-\alpha+1)\sum_{n=1}^{\infty}(-1)^{n-1}\frac{|z|^{n}}{1+\alpha (n-1)}.\eea
Using (\ref{e9}), we have
\beas|f(z)|=\left|\int_0^z\left(\frac{\partial f}{\partial \xi}d\xi+\frac{\partial f}{\partial \ol{\xi}}d\ol{\xi}\right)\right|&=&\left|\int_{0}^{z}\left(h^{\prime}(\xi)d\xi + \overline{g^{\prime}(\xi)}d\overline{\xi}\right)\right|\\[1mm]
&\geq& \int_0^{z}\left(|h'(\xi)|-|g'(\xi)|\right)|d\xi|\\[1mm]
&\geq&\int_0^{|z|}\left(1-2(M-\alpha+1)\sum_{n=1}^{\infty}(-1)^{n-1}\frac{t^{n}}{1+\alpha (n-1)}\right)dt\\[1mm]
&=&|z|+2(M-\alpha+1)\sum_{n=2}^{\infty}(-1)^{n-1}\frac{|z|^{n}}{n(1+\alpha (n-2))}.\eeas
{\bf Case 2.} Let $\alpha=1$. From (\ref{e5}), we have
\beas\frac{d}{dz}\left(F_{\epsilon}'(z)\right)&=&\frac{2M\omega(z)}{z(1-\omega(z))}.\eeas
Using $F_{\epsilon}'(0)=1$, we have
\beas |F_{\epsilon}'(z)|&=&\left|1+2M\int_{0}^{|z|}\frac{\omega(te^{i\theta})}{ te^{i\theta}(1-\omega(te^{i\theta}))} e^{i\theta}dt\right|\\
&\leq& 1+2M\int_{0}^{|z|}\frac{dt}{1-t}=1+2M\int_{0}^{|z|}\sum_{n=0}^\infty t^ndt=1+2M\sum_{n=1}^\infty \frac{|z|^{n}}{n}.
\eeas
Similarly, we obtain 
\beas |F_{\epsilon}'(z)|&\geq& 1+2M\int_{0}^{|z|}\frac{1}{t}\text{Re}\left(\frac{\omega(te^{i\theta})}{1-\omega(te^{i\theta})}\right) dt\\[2mm]
&\geq& 1-2M\int_{0}^{|z|}\frac{dt}{1+t}=1-2M\sum_{n=1}^\infty \frac{(-1)^{n-1}|z|^{n}}{n}.\eeas
Using similar argument as in \textrm{Case 1}, we have 
\beas |z|+2M\sum_{n=2}^{\infty}\frac{(-1)^{n-1}|z|^{n}}{n(n-1)}\leq |f(z)|\leq |z|+2M\sum_{n=2}^{\infty}\frac{|z|^n}{n(n-1)}.\eeas
Equality in (\ref{eq1}) holds when $f$ given by
\beas f(z)=z+2(M-\alpha+1)\sum_{n=2}^{\infty}\frac{z^n}{n(1+\alpha (n-2))}\in\mathcal{P}_{\mathcal{H}}^0(\alpha,M)\eeas
or its suitable rotations.
This completes the proof.
\end{proof}
\begin{cor} \label{CTheo1}
Let $f = h + \overline{g} \in \mathcal{P}_{\mathcal{H}}^{0}(\alpha,M)$ be of the form (1.1) for $M > 0$ and $\alpha \in (0,1]$. For $|z| = r < 1$, we have
\beas
|f(z)| &\leq& r + \frac{2(M-\alpha+1)}{\alpha} \left(r + (1-r)\log(1-r)\right)\\\text{and}\quad
|f(z)| &\geq&r - \frac{2(M-\alpha+1)}{\alpha} \left(r + (1-r)\log(1-r)\right).\eeas
\end{cor}
\begin{proof}
Note that for $n \geq 2$ and $\alpha \in (0,1]$, we have $1 + \alpha(n-2) \geq \alpha(n-1)$ and consequently,
\beas
\frac{1}{n(1+\alpha(n-2))} \leq \frac{1}{\alpha n(n-1)}.\eeas
Therefore, we have
\beas
\sum_{n=2}^{\infty} \frac{r^n}{n(1+\alpha(n-2))} \leq \frac{1}{\alpha} \sum_{n=2}^{\infty} \frac{r^n}{n(n-1)} = \frac{1}{\alpha}\left( r+(1-r) \log (1-r)\right).\eeas
For the lower bound, note that
\beas
\left|\sum_{n=2}^{\infty}\frac{(-1)^{n-1}r^{n}}{n(1+\alpha(n-2))}\right| \leq \sum_{n=2}^{\infty}\frac{r^{n}}{n(1+\alpha(n-2))} \leq \frac{1}{\alpha} \left(r + (1-r)\log(1-r)\right).\eeas
The result follows from \textrm{Theorem} \ref{Theo1}.
\end{proof}
\begin{rem} Setting $\alpha=1$ in \textrm{Theorem \ref{Theo1}} gives the sharp version of \textrm{Theorem C}. \end{rem}
After establishing the growth characteristics of functions in the class $\mathcal{P}_{\mathcal{H}}^0(\alpha,M)$, a subsequent geometric investigation involves determining the largest disk within which these functions are guaranteed to be univalent or convex. This is known as the radius problem.
We now recall the following well-known result.
\begin{lem}\label{lem6}\cite{AZ1990} Let $f=h+\ol{g}$ be given by (\ref{e1}). 
(i) If $\sum_{n=2}^{\infty} n\left(|a_n|+|b_n|\right)\leq 1$, then $f$ is starlike in $\mathbb{D}$; 
(ii) If $\sum_{n=2}^{\infty} n^2\left(|a_n|+|b_n|\right)\leq 1$, then $f$ is convex in $\mathbb{D}$.
\end{lem}
\begin{rem}
For parameters $a, b, c\in\mathbb{C}$ with $c\not=-k$ $(k\in\mathbb{N}\cup\{0\})$ and $z\in\mathbb{D}$, the Gaussian hypergeometric function is defined by
\beas F(a, b, c, z):={}_2F_1(a,b;c;z)=\sum_{n=0}^\infty \gamma_n z^n,\quad\text{where}\quad \gamma_n=\frac{(a)_n(b)_n}{(c)_n(1)_n},\eeas
which converges for all $z\in\mathbb{D}$ and converges on the circle $|z|=1$ if $\text{Re}(c-a-b)>0$. Here, $(a)_n$ is the Pochhammer symbol.
It is evident that the Gaussian hypergeometric function $F(a,b;c;z)$ is analytic within the domain $\mathbb{D}$, and it can also be analytically continued outside the unit circle.
For further details, we refer to \cite{QV2005} and the references cited therein.\end{rem}
\noindent In view of \textrm{Theorem \ref{T3}} and \textrm{Lemma \ref{lem6}}, it is possible to show that each $f\in\mathcal{P}_{\mathcal{H}}^0(\alpha,M)$ is convex (resp. starlike) in some disk $D$.
\begin{theo}\label{Theo2} Let $f=h+\ol{g}\in\mathcal{P}_{\mathcal{H}}^0(\alpha,M)$ be given by (\ref{e1}) for $M>0$ and $\alpha\in(0,1]$. Then, \begin{enumerate}
\item[(i)] $f$ is starlike in $|z|<r^*$, where $r^*\in(0,1)$ is the smallest root of the equation 
\bea\label{a1} G_{1,\alpha,M}(r):=r\;{}_2F_1\left(1,\frac{1}{\alpha};1+\frac{1}{\alpha}; r\right)-\frac{1}{2(M-\alpha+1)}=0;\eea
\item[(ii)] $f$ is convex in $|z|<r_c$, where $r_c\in(0,1)$ is the smallest root of the equation 
\beas G_{2,\alpha,M}(r):=\frac{r}{1-r}+\left(2\alpha-1\right)r\;{}_2F_1\left(1,\frac{1}{\alpha};1+\frac{1}{\alpha}; r\right)-\frac{\alpha}{2(M-\alpha+1)}=0.\eeas
\end{enumerate}\end{theo}
\begin{proof} Consider the dilation $f_r(z)=(1/r)f(rz)=z+\sum_{n=2}^{\infty} a_nr^{n-1}z^n+\ol{\sum_{n=2}^{\infty} b_nr^{n-1}z^n}$ for 
$z\in\mathbb{D}$ and $0<r<1$.
For convenience, let us consider the sums:
\beas S_1(r)=\sum_{n=2}^{\infty} n\left(|a_n|+|b_n|\right)r^{n-1}\quad\text{and}\quad S_2(r)=\sum_{n=2}^{\infty} n^2\left(|a_n|+|b_n|\right)r^{n-1}.\eeas
(i) Applying \textrm{Theorem \ref{T3}}, we obtain
\beas S_1\leq 2(M-\alpha+1)\sum_{n=2}^{\infty}\frac{r^{n-1}}{1+\alpha (n-2)}&=&\frac{2(M-\alpha+1)}{\alpha}\sum_{n=2}^{\infty}\frac{r^{n-1}}{\frac{1}{\alpha}+n-2}\\[2mm]
&=&\frac{2(M-\alpha+1)}{\alpha}\left(r^{1-\frac{1}{\alpha}}\sum_{n=2}^{\infty}\int_{t=0}^r t^{\frac{1}{\alpha}+n-3}dt\right)\\[1.5mm]
&=&\frac{2(M-\alpha+1)}{\alpha}\left(r^{1-\frac{1}{\alpha}}\int_{t=0}^r \frac{t^{\frac{1}{\alpha}-1}}{1-t}dt\right)\\
&=&\frac{2(M-\alpha+1)}{\alpha}r\int_{s=0}^1 \frac{s^{\frac{1}{\alpha}-1}}{1-rs}ds.\eeas
Using the Gaussian hypergeometric function, we obtain the following identity:
\bea \label{tt1}\int_{s=0}^1 \frac{s^{\frac{1}{\alpha}-1}}{1-rs}ds=\alpha\;{}_2F_1\left(1,\frac{1}{\alpha};1+\frac{1}{\alpha}; r\right).\eea
Thus, we have
\beas S_1(r)\leq 2(M-\alpha+1)r\;{}_2F_1\left(1,\frac{1}{\alpha};1+\frac{1}{\alpha}; r\right)\leq1\quad\text{for}\quad |z|\leq r^*,\eeas
where $r^*\in(0,1)$ is the smallest root of the equation 
\bea\label{t1} G_{1,\alpha,M}(r):=r\;{}_2F_1\left(1,\frac{1}{\alpha};1+\frac{1}{\alpha}; r\right)-\frac{1}{2(M-\alpha+1)}=0.\eea
Note that ${}_2F_1\left(1,\frac{1}{\alpha};1+\frac{1}{\alpha}; 0\right)=1$, ${}_2F_1\left(1,\frac{1}{\alpha};1+\frac{1}{\alpha}; 1\right)=+\infty$ and the function 
$G_{1,\alpha,M}(r)$ is continuous in $(0,1)$ with $\lim_{r\to 0^+} G_{1,\alpha,M}(r)=-\frac{1}{2(M-\alpha+1)}<0$ and $\lim_{r\to1^-}G_{1,\alpha,M}(r)=+\infty$. By the 
Intermediate Value Theorem, there exists an $r^*\in(0,1)$ such that $G_{1,\alpha,M}(r^*)=0$. Since $S_1(r) \leq 1$ for $r \leq r^*$, by \textrm{Lemma \ref{lem6}} (i), it foll0ws that the dilation $f_r$ is starlike in $\mathbb{D}$ for $r \leq r^*$. 
We know that, if $f_r$ is starlike in $\mathbb{D}$, then the image $f_r(\mathbb{D}) = (1/r)f(r\mathbb{D})$ is a starlike domain. Since scaling preserves starlikeness, $f(r\mathbb{D})$ 
is also starlike. Therefore, $f$ is starlike in the disk $|z|< r$ for $r \leq r^*$, and consequently, $f$ is starlike in $|z| < r^*$, where $r^*\in(0,1)$ is the smallest root of (\ref{t1}).\\[2mm]
\noindent (ii) Combining \textrm{Theorem \ref{T3}} with the identity (\ref{tt1}), we obtain  
\beas S_2(r)&\leq&2(M-\alpha+1)\sum_{n=2}^{\infty}\frac{nr^{n-1}}{1+\alpha (n-2)}\hspace{3cm}\\[1mm]
&=&\frac{2(M-\alpha+1)}{\alpha}\left(\sum_{n=2}^{\infty}\frac{r^{n}}{\frac{1}{\alpha}+n-2}\right)'\\[1mm]
&=&\frac{2(M-\alpha+1)}{\alpha}\left(r^{2-\frac{1}{\alpha}}\sum_{n=2}^{\infty}\int_{t=0}^r t^{\frac{1}{\alpha}+n-3}dt\right)'\\[1mm]
&=&\frac{2(M-\alpha+1)}{\alpha}\left(r^{2-\frac{1}{\alpha}}\int_{t=0}^r \frac{t^{\frac{1}{\alpha}-1}}{1-t}dt\right)'\\[1mm]
&=&\frac{2(M-\alpha+1)}{\alpha}\left(\frac{r}{1-r}+\left(2-\frac{1}{\alpha}\right)r^{1-\frac{1}{\alpha}}\int_{t=0}^r \frac{t^{\frac{1}{\alpha}-1}}{1-t}dt\right)\\[1mm]
&=&\frac{2(M-\alpha+1)}{\alpha}\left(\frac{r}{1-r}+\left(2-\frac{1}{\alpha}\right)r\int_{s=0}^1 \frac{s^{\frac{1}{\alpha}-1}}{1-rs}ds\right)\\[1mm]
&=&\frac{2(M-\alpha+1)}{\alpha}\left(\frac{r}{1-r}+\left(2\alpha-1\right)r\;{}_2F_1\left(1,\frac{1}{\alpha};1+\frac{1}{\alpha}; r\right)\right)\\[1mm]\
&\leq& 1\quad\text{for}\quad |z|\leq r_c,\eeas
where $r_c\in(0,1)$ is the smallest root of the equation 
\bea\label{t2} G_{2,\alpha,M}(r):=\frac{r}{1-r}+\left(2\alpha-1\right)r\;{}_2F_1\left(1,\frac{1}{\alpha};1+\frac{1}{\alpha}; r\right)-\frac{\alpha}{2(M-\alpha+1)}=0.\eea
Note that the function $G_{2,\alpha,M}(r)$ is continuous in $(0,1)$ with $\lim_{r\to 0^+} G_{2,\alpha,M}(r)=-\frac{\alpha}{2(M-\alpha+1)}<0$ and $\lim_{r\to 1^-} G_{2,\alpha,M}(r)=+\infty$. In view of the Intermediate value theorem, there exists a $r_c\in(0,1)$ such that $G_{2,\alpha,M}(r_c)=0$. Since $S_2(r) \leq 1$ for $r \leq r_c$, by 
\textrm{Lemma \ref{lem6}} (ii), it follows that the dilation $f_r$ is convex in $\mathbb{D}$ for $r \leq r^*$. 
If $f_r$ is convex in $\mathbb{D}$, then the image $f_r(\mathbb{D}) = (1/r)f(r\mathbb{D})$ is a convex domain. Since scaling preserves convexity, $f(r\mathbb{D})$ 
is also convex. Therefore, $f$ is convex in the disk $|z| < r$ for $r \leq r_c$, and consequently, $f$ is convex in $|z| < r_c$, where $r_c\in(0,1)$ is the smallest root of (\ref{t2}).
This completes the proof.\end{proof}
\begin{rem} In view of \textrm{Theorem \ref{Theo2}}, each function belonging to the class $\mathcal{P}_{\mathcal{H}}^0(\alpha,M)$ is univalent within the disk $|z|< r^*$, where $r^*\in(0, 1)$ denotes the smallest root of the equation (\ref{a1}).
\end{rem}
In Table \ref{tab1}, we obtain the values of $r^*$ and $r_c$ for specific values of $M>0$ and $\alpha\in(0,1]$.
\begin{table}[H]
\centering
\caption{$r^*$ (resp. $r_c$) is the smallest root of the equation (\ref{t1}) (resp. (\ref{t2})) in $(0,1)$}
\label{tab1}
\begin{tabular}{*{8}{|c}|}
\hline
$\alpha$&1		&1/2			&1/2			&1/3			&3/4			&3/4			&1/4\\
\hline
$M$&1/2				&1/4			&2				&1/9			&1/100		&1/10		&1/5\\
\hline
$r^*$&0.632121	&0.454395	&0.176134	&0.42966	&0.823912 &0.732081 	&0.368607\\
\hline
$r_c$&0.357799	&0.25&0.0909091	&0.237029	&0.515173 &0.436194 	&0.200939\\
\hline
\end{tabular}
\end{table}
\begin{rem} Table \ref{tab1} clearly illustrates that for a fixed $\alpha$, there is an inverse relationship between the parameter $M$ and the radius of univalency $r^*$. As $M$ increases, it allows functions to exhibit greater growth, which subsequently reduces the assured radius of univalency for the entire class. \end{rem}
\begin{rem}It is easy to observe that $r_c < r^{*}$ for all parameter combinations, which aligns perfectly with the geometric hierarchy that convexity is a stronger condition than starlikeness. The computed radii provide boundaries within which all functions in $\mathcal{P}_{\mathcal{H}}^{0}(\alpha,M)$ are guaranteed to possess these geometric properties.\end{rem}
In Figures \ref{Fig1} and \ref{Fig2}, we illustrate the locations of $r^*$ and $r_c$ for certain values of $M>0$ and $\alpha\in(0,1]$.
\begin{figure}[H]
\centering
\includegraphics[scale=0.9]{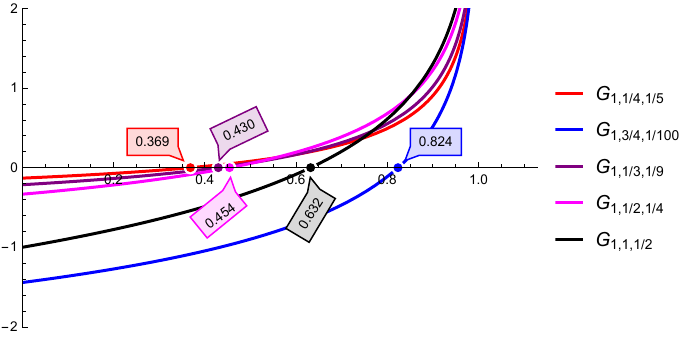}
\caption{Graph of $G_{1,\alpha,M}(r)$ for different values of $M>0$ and $\alpha\in(0,1]$}
\label{Fig1}
\end{figure}
\begin{figure}[H]
\centering
\includegraphics[scale=0.9]{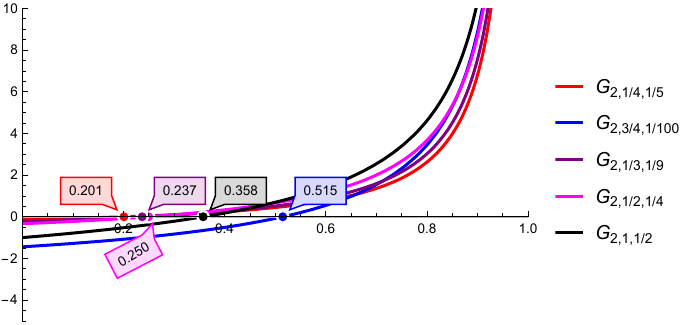}
\caption{Graph of $G_{2,\alpha,M}(r)$ for different values of $M>0$ and $\alpha\in(0,1]$}
\label{Fig2}
\end{figure}
\begin{rem} Figures 1 and 2 shows the behaviour of the functions $G_{1,\alpha,M}(r)$ and $G_{2,\alpha,M}(r)$, whose roots are the radii $r^*$ and $r_c$, respectively. The graphs confirm the existence and uniqueness of these roots within the interval (0, 1) and demonstrate their dependence on the parameters. 
\end{rem}
\noindent In the following result, we establish a sufficient condition for harmonic functions to belong to $\mathcal{P}_{\mathcal{H}}^0(\alpha,M)$.
\begin{theo}\label{TT1} Let $f=h+\ol{g}\in\mathcal{H}$ with $g'(0)=0$ be given by (\ref{e1}). If 
\beas \sum_{n=2}^\infty (n+\alpha n(n-2))\left(|a_n|+|b_n|\right)< M-\alpha+1,\eeas
then $f\in\mathcal{P}_{\mathcal{H}}^0(\alpha,M)$.\end{theo}
\begin{proof}Since $h(z)=z+\sum_{n=2}^\infty a_nz^n$ and $g(z)=\sum_{n=2}^\infty b_nz^n$, thus, it follows that
\beas \textrm{Re}\left((1-\alpha)h'(z)+\alpha zh''(z)\right)&=&\text{Re}\left((1-\alpha)+\sum_{n=2}^\infty (n+\alpha n(n-2))a_nz^{n-1}\right)\\[1mm]
&\geq&(1-\alpha)-\left|\sum_{n=2}^\infty (n+\alpha n(n-2))a_nz^{n-1}\right|\\[1mm]
&\geq&(1-\alpha)-\sum_{n=2}^\infty (n+\alpha n(n-2))|a_n|.\eeas
Given that $\sum_{n=2}^\infty (n+\alpha n(n-2))\left(|a_n|+|b_n|\right)< M-\alpha+1$, thus, we have
\beas \text{Re}\left((1-\alpha)h'(z)+\alpha zh''(z)\right)&>&  -M+\sum_{n=2}^\infty  (n+\alpha n(n-2))|b_n|\\[1mm]
&\geq& -M+\left|\sum_{n=2}^\infty  (n+\alpha n(n-2))b_n z^{n-1}\right|\\[1mm]
&=&-M+\left|(1-\alpha)g'(z)+\alpha zg''(z)\right|,\eeas
which shows that $f\in\mathcal{P}_{\mathcal{H}}^0(\alpha,M)$. This completes the proof.
\end{proof}
\section{Structural Properties: Convex Combinations and Convolutions}
\noindent In the following result, we establish that the class $\mathcal{P}_{\mathcal{H}}^0(\alpha,M)$ is invariant under convex combinations.
\begin{theo}\label{T5} The class $\mathcal{P}_{\mathcal{H}}^0(\alpha,M)$ is closed under convex combinations.\end{theo}
\begin{proof} Let $f_i=h_i+\ol{g_i}\in\mathcal{P}_{\mathcal{H}}^0(\alpha,M)$ for $1\leq i\leq n$ and $\sum_{i=1}^n t_i=1$, where $0\leq t_i\leq 1$ for each $i$. Thus, we have 
\beas \text{Re}\left((1-\alpha)h_i'(z)+ \alpha zh_i''(z)\right)>-M+\left|(1-\alpha)h_i'(z)+\alpha zg_i''(z)\right|\eeas
with $h_i(0)=g_i(0)=h_i'(0)-1=g_i'(0)=0$ for $1\leq i\leq n$ and $z\in\mathbb{D}$.
The convex combination of the $f_i$'s can be written as $f(z)=\sum_{i=1}^n t_if_i(z)=h(z)+\ol{g(z)}$,
where $h(z)=\sum_{i=1}^n t_i h_i(z)$ and $g(z)=\sum_{i=1}^n t_ig_i(z)$. It is evident that both $h$ and $g$ are analytic in $\mathbb{D}$ with $h(0)=g(0)=h'(0)-1=g'(0)=0$.
Therefore,
\beas \text{Re}\left((1-\alpha)h'(z)+ \alpha zh''(z)\right)&=&\text{Re}\left((1-\alpha)\sum_{i=1}^n t_i h_i'(z)+\alpha z\sum_{i=1}^nt_ih_i''\right)\hspace{2cm}\\[2mm]
&=&\text{Re}\left(\sum_{i=1}^n t_i\left((1-\alpha)h_i'(z)+ \alpha zh_i''(z)\right)\right)\\[1mm]
&>&\sum_{i=1}^nt_i \left(-M+\left|(1-\alpha)g_i'(z)+\alpha zg_i''(z)\right|\right)\eeas
\beas\hspace{4cm}&=&-M+\sum_{i=1}^nt_i\left|(1-\alpha)g_i'(z)+\alpha zg_i''(z)\right| \\[1mm]
&\geq&-M+\left|(1-\alpha)\left(\sum_{i=1}^nt_ig_i'(z)\right)+\alpha z\left(\sum_{i=1}^nt_i g_i''(z)\right)\right|\\[1mm]
&=& -M+\left|(1-\alpha)g'(z)+\alpha zg''(z)\right|,\eeas
which shows that $f\in\mathcal{P}_{\mathcal{H}}^0(\alpha,M)$. This completes the proof.\end{proof}
\noindent Before we proceed to the convolution theorem for the harmonic class, we first establish an essential structural property of the related analytic class $\mathcal{P}(\alpha, M)$, namely its convexity.
\begin{lem}\label{lem1}
The class \(\mathcal{P}(\alpha, M)\) is closed under convex combinations.
\end{lem}
\begin{proof}
Let $F_1, F_2, \ldots, F_n \in \mathcal{P}(\alpha, M)$. Consider the convex combination $F=\sum_{i=1}^n t_i F_i$, where $t_i\geq 0$ for each $i$ with $\sum_{i=1}^n t_i=1$.
Thus, for each $i$, we have
\beas \text{Re}\left((1-\alpha)F_i'(z) + \alpha z F_i''(z)\right) > -M \quad \text{for }\quad z \in \mathbb{D}.\eeas
Therefore.
\beas
\text{Re}\left((1-\alpha)F'(z) + \alpha z F''(z)\right)
&= & \text{Re}\left(\sum_{i=1}^n t_i \left((1-\alpha)F_i'(z) + \alpha z F_i''(z)\right)\right) \\[1mm]
&=& \sum_{i=1}^n t_i  \text{Re}\left((1-\alpha)F_i'(z) + \alpha z F_i''(z)\right) \\[1mm]
&>& \sum_{i=1}^n t_i (-M) = -M,
\eeas
which shows that $F \in \mathcal{P}(\alpha, M)$. This completes the proof.
\end{proof}
\noindent We need the following essential lemmas to prove our convolution results.
\begin{lem}\label{lem2}\cite{SS1989} Let $\{a_n\}$ be a convex null sequence. Then the function $p$ given by $p(z)=a_0/2+\sum_{n=1}^{\infty} a_nz^n$ is analytic in 
$\mathbb{D}$ and $\text{Re}(p(z))>0$ in $\mathbb{D}$. \end{lem}
\begin{lem}\label{lem3}\cite{SS1989} Let the function $p$ be analytic in $\mathbb{D}$ with $p(0)=1$ and $\text{Re}(p(z))>1/2$ in $\mathbb{D}$. Then for any analytic function $f$ 
in $\mathbb{D}$, the function $p*f$ takes values in the convex hull of the image of $\mathbb{D}$ under $f$. \end{lem}
\begin{lem}\label{lem4} Let $F\in\mathcal{P}(\alpha, M)$ with $M-\alpha+1\leq 3(1+\alpha)/(6\alpha+4)$. Then, $\text{Re}\left(F(z)/z\right)>1/2$ for $z\in\mathbb{D}$.\end{lem}
\begin{proof} Let $F\in\mathcal{P}(\alpha, M)$ be given by $F(z)=z+\sum_{n=2}^\infty A_nz^n$. Then, we have 
\beas &&\qquad\quad\text{Re}\left((1-\alpha)F'(z)+ \alpha zF''(z)\right)>-M,\\ 
&& \text{\it i.e.,}\quad\text{Re}\left(1+\frac{1}{2(M-\alpha+1)}\sum_{n=2}^\infty (n+\alpha n(n-2))A_nz^{n-1}\right)>\frac{1}{2}
\quad\text{for}\quad z\in\mathbb{D}.\eeas
Let $p(z)=1+\frac{1}{2(M-\alpha+1)}\sum_{n=2}^\infty (n+\alpha n(n-2))A_nz^{n-1}$. Then, $p(0)=1$ and $\text{Re}(p(z))>1/2$ in $\mathbb{D}$.
We consider a sequence $\{c_n\}$ defined by $c_0=1$ and $c_{n-1}=\frac{2(M-\alpha+1)}{n+\alpha n(n-2)}$ for $n\geq 2$. It is evident that $c_n\to 0$ as $n\to\infty$. Note that $c_0-c_1=1-(M-\alpha+1)$, $c_1-c_2=(M-\alpha+1)(1+3\alpha)/(3+3\alpha)$, $c_2-c_3=(5 \alpha +1)(M-\alpha +1)/6( (\alpha +1) (2 \alpha +1))$, $c_3-c_4=(7 \alpha +1)(M-\alpha +1)/(10 (2 \alpha +1) (3 \alpha +1))$, etc. 
It is evident that 
\beas &&(c_1-c_2)-(c_2-c_3)=\frac{\left(12 \alpha ^2+5 \alpha +1\right)(M-\alpha +1)}{6(\alpha +1) (2 \alpha +1)}\geq 0,\\
&&(c_2-c_3)-(c_3-c_4)=\frac{\left(27 \alpha ^2+8 \alpha +1\right)(M-\alpha +1)}{15(\alpha +1) (2 \alpha +1) (3 \alpha +1)}\geq 0\eeas
and so on. Thus, we have $c_{i-2}-c_{i-1}\geq c_{i-1}-c_{i}\geq 0$ for $i\geq 3$.
Therefore, 
\beas c_0-c_1\geq c_1-c_2\geq\cdots\geq c_{n-1}-c_n\geq\cdots\geq 0\eeas 
is possible only when $M-\alpha+1\leq 3(1+\alpha)/(6\alpha+4)$.
Thus, $\{c_n\}$ is a convex null sequence. In view of \textrm{Lemma \ref{lem2}}, we have the function 
\beas q(z)=\frac{1}{2}+\sum_{n=2}^\infty \frac{2(M-\alpha+1)}{n+\alpha n(n-2)}z^{n-1},\eeas
which is analytic in $\mathbb{D}$ and $\text{Re} (q(z))>0$ in $\mathbb{D}$. Thus, we have
\bea\label{ew3}\frac{F(z)}{z}=1+\sum_{n=2}^\infty A_nz^{n-1}=p(z)*\left(1+\sum_{n=2}^\infty \frac{2(M-\alpha+1)}{n+\alpha n(n-2)}z^{n-1}\right)
.\eea
In view of \textrm{Lemma \ref{lem3}} and using (\ref{ew3}), we have $\text{Re}\left(F(z)/z\right)>1/2$ for $z\in\mathbb{D}$. This completes the proof.
\end{proof}
\begin{lem}\label{lem5} Let $F_1,F_2\in\mathcal{P}(\alpha, M)$ with $M-\alpha+1\leq 3(1+\alpha)/(6\alpha+4)$. Then, $F_1*F_2\in\mathcal{P}(\alpha, M)$.\end{lem}
\begin{proof} Let $F_1(z)=z+\sum_{n=2}^\infty A_nz^n$ and $F_2(z)=z+\sum_{n=2}^\infty B_nz^n$. The convolution of $F_1$ and $F_2$ is given by $F(z)=F_1(z)*F_2(z)=z+\sum_{n=2}^\infty A_nB_nz^n$.
Therefore,
\bea\label{ew4} (1-\alpha)F'(z)+ \alpha zF''(z)&=&(1-\alpha)+\sum_{n=2}^\infty (n+\alpha n(n-2))A_nB_nz^{n-1}\nonumber\\
&=&\left(\frac{F_2(z)}{z}\right)*\left((1-\alpha)F_1'(z)+ \alpha zF_1''(z)\right).\eea
Since $F_1, F_2\in\mathcal{P}(\alpha, M)$, it follows that $\text{Re}\left((1-\alpha)F_1'(z)+ \alpha zF_1''(z)\right)>-M$ and in view of \textrm{Lemma \ref{lem4}}, we have $\text{Re}\left(F_2(z)/z\right)>1/2$ for $z\in\mathbb{D}$.
In view of \textrm{Lemma \ref{lem3}} and using (\ref{ew4}), we have $\text{Re}\left((1-\alpha)F(z)+ \alpha zF''(z)\right)>-M$ for $z\in\mathbb{D}$. Hence, $F=F_1*F_2\in\mathcal{P}(\alpha, M)$. This completes the proof.
 \end{proof}
  In the following result, we establish that the class $\mathcal{P}_{\mathcal{H}}^0(\alpha,M)$ is invariant under convolutions for some certain condition.
\begin{theo}\label{CTh2} Let $F_1,F_2\in\mathcal{P}_{\mathcal{H}}^0(\alpha,M)$ with $M-\alpha+1\leq 3(1+\alpha)/(6\alpha+4)$. Then, $F_1*F_2\in\mathcal{P}_{\mathcal{H}}^0(\alpha,M)$.\end{theo}
\begin{proof} Let $F_1=h_1+\ol{g_1}$ and $F_2=h_2+\ol{g_2}$ be two functions in $\mathcal{P}_{\mathcal{H}}^0(\alpha,M)$. The convolution of $F_1$ and $F_2$ is given by
$F_1*F_2=h_1*h_2+\ol{g_1*g_2}$. In order to show that $F_1*F_2\in\mathcal{P}_{\mathcal{H}}^0(\alpha,M)$, it suffices to prove that the function
$F_\epsilon=h_1*h_2+\epsilon\left( g_1*g_2\right)\in\mathcal{P}(\alpha, M)$ for each $\epsilon$ with $|\epsilon|=1$. 
By \textrm{Theorem \ref{T1}}, we have $h_1+\epsilon g_1, h_2+\epsilon g_2\in\mathcal{P}(\alpha, M)$ for each $\epsilon$ with $|\epsilon|=1$. Therefore, we have
\beas F_\epsilon=h_1*h_2+\epsilon\left( g_1*g_2\right)=\frac{1}{2}\left((h_1-g_1)*(h_2-\epsilon g_2)\right)+\frac{1}{2}\left((h_1+g_1)*(h_2+\epsilon g_2)\right).\eeas
Using \textrm{Lemma \ref{lem5}}, we have $(h_1-g_1)*(h_2-\epsilon g_2), (h_1+g_1)*(h_2+\epsilon g_2)\in\mathcal{P}(\alpha, M)$ for $M-\alpha+1\leq 3(1+\alpha)/(6\alpha+4)$. 
Since $\mathcal{P}(\alpha, M)$ is a convex set by \textrm{Lemma} \ref{lem1}, the convex combination $F_\epsilon \in\mathcal{P}(\alpha, M)$ for each $\epsilon$ with $|\epsilon|=1$. By \textrm{Theorem} \ref{T1}, we conclude that 
$F_1 * F_2 \in \mathcal{P}^{0}_{\mathcal{H}}(\alpha, M)$. This completes the proof.
 \end{proof}
 \begin{rem}
The condition $M-\alpha+1\leq 3(1+\alpha)/(6\alpha+4)$ in Lemma \ref{lem4} ensures that the sequence $\{c_n\}$ is a convex null sequence. This is essential for applying the classical result of Singh and Singh \cite{SS1989}. This condition defines a specific subfamily of $\mathcal{P}_{\mathcal{H}}^0(\alpha,M)$, where the convolution properties hold. Investigating the convolution for the entire class remains an open problem.
\end{rem}
In 2002, Goodloe \cite{G2002} defined the Hadamard product of a harmonic function with an analytic function as: $f\tilde{*}\phi=h*\phi+\ol{g*\phi}$,
where $f=h+\ol{g}$ is harmonic mapping in $\mathbb{D}$ and $\phi$ is an analytic function in $\mathbb{D}$.
\begin{theo}\label{Th1}
Let $f\in\mathcal{P}_{\mathcal{H}}^0(\alpha,M)$ be given by (\ref{e1}) and $\phi\in\mathcal{A}$ be such that $\text{Re}\left(\phi(z)/z\right)$ $>1/2$ for $z\in\mathbb{D}$. Then, 
$f\tilde{*}\phi\in\mathcal{P}_{\mathcal{H}}^0(\alpha,M)$.\end{theo}
\begin{proof} Let $f=h+\ol{g}\in\mathcal{P}_{\mathcal{H}}^0(\alpha,M)$. In view of \textrm{Theorem \ref{T1}}, we have $f_\epsilon=h+\epsilon g\in\mathcal{P}(\alpha, M)$ for each 
$\epsilon$ with $|\epsilon|=1$. To show that $f\tilde{*}\phi=h*\phi+\ol{g*\phi}\in\mathcal{P}_{\mathcal{H}}^0(\alpha,M)$, it suffices to prove that the function
$F_\epsilon(z)=h*\phi+\epsilon(g*\phi)\in\mathcal{P}(\alpha, M)$ for each $\epsilon$ $(|\epsilon|=1)$. Since $f_\epsilon\in\mathcal{P}(\alpha, M)$ and $\phi\in\mathcal{A}$, let $f_\epsilon(z)=z+\sum_{n=2}^\infty A_nz^n$ and $\phi(z)=z+\sum_{n=2}^\infty B_nz^n$.
Therefore, $F_\epsilon=f_\epsilon*\phi=z+\sum_{n=2}^\infty A_nB_nz^n$.
Thus, we have
\bea \label{ew5}(1-\alpha)F_\epsilon'(z)+ \alpha zF_\epsilon''(z)&=&(1-\alpha)+\sum_{n=2}^\infty (n+\alpha n(n-2))A_nB_nz^{n-1}\nonumber\\
&=&\left(\frac{\phi(z)}{z}\right)*\left((1-\alpha)f_\epsilon'(z)+ \alpha zf_\epsilon''(z)\right).\eea
As $\text{Re}\left(\phi(z)/z\right)>1/2$ for $z\in\mathbb{D}$ and $f_\epsilon\in\mathcal{P}(\alpha, M)$, {\it i.e.,} $\text{Re}\left((1-\alpha)f_\epsilon'(z)+ \alpha zf_\epsilon''(z)\right)>-M$, in view of \textrm{Lemma \ref{lem3}} and using (\ref{ew5}), we have 
\beas \text{Re}\left((1-\alpha)F_\epsilon'(z)+ \alpha zF_\epsilon''(z)\right)>-M\quad \text{for}\quad  z\in\mathbb{D},\eeas
which shows that $F_\epsilon\in\mathcal{P}(\alpha, M)$. This completes the proof.
\end{proof}
\begin{cor} Let $f\in\mathcal{P}_{\mathcal{H}}^0(\alpha,M)$ be given by (\ref{e1}) and $\phi\in\mathcal{K}$, where $\mathcal{K}$ denotes the family of all convex functions in $\mathbb{D}$. Then, 
$f\tilde{*}\phi\in\mathcal{P}_{\mathcal{H}}^0(\alpha,M)$.\end{cor}
\begin{proof} Since $\phi\in\mathcal{K}$, it follows that $\text{Re}\left(\phi(z)/z\right)>1/2$ for $z\in\mathbb{D}$. The result immediately follows from \textrm{Theorem \ref{Th1}}.\end{proof}
\begin{theo}\label{Th2}
Let $f\in\mathcal{P}_{\mathcal{H}}^0(\alpha,M)$ be given by (\ref{e1}) and $\phi\in\mathcal{A}$ be such that $\text{Re}\left(\phi(z)/z\right)$ $>1/2$ for $z\in\mathbb{D}$. Then, 
$f*\left(\phi+\beta\ol{\phi}\right)\in\mathcal{P}_{\mathcal{H}}^0(\alpha,M)$, where $|\beta|=1$.\end{theo}
\begin{proof} Let $f=h+\ol{g}\in\mathcal{P}_{\mathcal{H}}^0(\alpha,M)$ be of the form (\ref{e1}), {\it i.e.,} $h(z)=z+\sum_{n=2}^\infty a_nz^n$, $g(z)=\sum_{n=2}^\infty b_nz^n$. Thus, we have
\beas f*\left(\phi+\beta\ol{\phi}\right)=\left(h+\ol{g}\right)*\left(\phi+\ol{\ol{\beta}\phi}\right)=h*\phi+\ol{\ol{\beta}(g*\phi)}.\eeas
To prove that $f*\left(\phi+\beta\ol{\phi}\right)\in\mathcal{P}_{\mathcal{H}}^0(\alpha,M)$, in view of \textrm{Theorem \ref{T1}}, it is sufficient to show that $f_{\epsilon}=h*\phi+\epsilon\ol{\beta}(g*\phi)\in\mathcal{P}(\alpha, M)$ for each $\epsilon$ $(|\epsilon|=1)$. Let $\phi(z)=z+\sum_{n=2}^\infty C_nz^n$. For each $|\epsilon|=1$, we have
\be\label{ew6}(1-\alpha)f_{\epsilon}'(z)+ \alpha zf_{\epsilon}''(z)=\left(\frac{\phi(z)}{z}\right)*\left((1-\alpha)\left(h(z)+\epsilon\ol{\beta}g(z)\right)'+\alpha z\left(h(z)+\epsilon\ol{\beta}g(z)\right)''\right).\ee
Since $f=h+\ol{g}\in\mathcal{P}_{\mathcal{H}}^0(\alpha,M)$, in view of \textrm{Theorem \ref{T1}}, we have $h+\epsilon \ol{\beta} g\in\mathcal{P}(\alpha, M)$ for $\epsilon$, 
$\beta$ with $|\epsilon\ol{\beta}|=1$, {\it i.e.}, $|\beta|=1$. Therefore, we have
 \beas\text{Re}\left((1-\alpha)\left(h(z)+\epsilon\ol{\beta}g(z)\right)'+\alpha z\left(h(z)+\epsilon\ol{\beta}g(z)\right)''\right)>-M\quad\text{for}\quad z\in\mathbb{D}.\eeas
Given that $\text{Re}\left(\phi(z)/z\right)>1/2$ for $z\in\mathbb{D}$, it follows that by using \textrm{Lemma \ref{lem3}} 
and (\ref{ew6}), we have 
\beas \text{Re}\left((1-\alpha)f_{\epsilon}'(z)+ \alpha zf_{\epsilon}''(z)\right)>-M\quad \text{for}\quad z\in\mathbb{D}.\eeas
Hence, $f_{\epsilon}\in\mathcal{P}(\alpha, M)$. This completes the proof.
\end{proof}
\begin{cor} Let $f\in\mathcal{P}_{\mathcal{H}}^0(\alpha,M)$ be given by (\ref{e1}) and $\phi\in\mathcal{K}$, where $\mathcal{K}$ denotes the family of all convex functions in $\mathbb{D}$. Then, 
$f*\left(\phi+\beta\ol{\phi}\right)\in\mathcal{P}_{\mathcal{H}}^0(\alpha,M)$, where $|\beta|=1$.\end{cor}
\begin{proof} Since $\phi\in\mathcal{K}$, it follows that $\text{Re}\left(\phi(z)/z\right)>1/2$ for $z\in\mathbb{D}$. The result immediately follows from \textrm{Theorem \ref{Th2}}.\end{proof}
 \section*{Declarations}
\noindent {\bf Acknowledgement:} The work of the second author is supported by the University Grants Commission (IN) fellowship (No. F. 44 - 1/2018 (SA - III)).\\
{\bf Conflict of Interest:} Authors declare that they have no conflict of interest.\\
{\bf Data availability:} Not applicable.\\
{\bf Authors' contributions:} All authors have equal contribution to complete the manuscript. All of them read and approved the final manuscript.

\end{document}